\documentclass[12pt, reqno]{amsart}
\usepackage{amsmath, amsthm, amscd, amsfonts, amssymb, graphicx, color}
\usepackage[bookmarksnumbered, colorlinks=false, plainpages]{hyperref}

\newtheorem{theorem}{Theorem}[section]
\newtheorem{lemma}[theorem]{Lemma}

\newtheorem{corollary}[theorem]{Corollary}
\theoremstyle{definition}
\newtheorem{definition}[theorem]{Definition}

\theoremstyle{remark}
\newtheorem{remark}[theorem]{Remark}
\numberwithin{equation}{section}

\allowdisplaybreaks
\begin{document}

\title[Infinite Matrix Operators and $E$-Frames in Hilbert Spaces]
{Infinite Matrix Operators and $E$-Frames in Hilbert Spaces}

\author[H.Hedayatirad]{Hassan Hedayatirad}
\address{Hassan Hedayatirad \\ Department of Mathematics and Computer
Sciences, Hakim Sabzevari University, Sabzevar, P.O. Box 397, IRAN}
\email{ \rm hasan.hedayatirad@hsu.ac.ir; hassan.hedayatirad67@gmail.com}
\author[T.L. Shateri]{Tayebe Lal Shateri }
\address{Tayebe Lal Shateri \\ Department of Mathematics and Computer
Sciences, Hakim Sabzevari University, Sabzevar, P.O. Box 397, IRAN}
\email{ \rm  t.shateri@hsu.ac.ir; t.shateri@gmail.com}
\thanks{*The corresponding author:
t.shateri@hsu.ac.ir; t.shateri@gmail.com (Tayebe Lal Shateri)}
 \maketitle
\begin{abstract}
The concept of an $E$-frame, recently introduced in frame theory, is obtained by applying an infinite invertible complex matrix to a sequence of elements of a Hilbert space $\mathcal{H}$. Here, $E$ is considered as a matrix mapping on the sequence space $\bigoplus_{n=1}^{\infty}\mathcal{H}$. A natural question is to determine the conditions on a sequence $\Psi=\{\psi_k\}_{k=1}^{\infty}$ and a matrix $E$ under which $E\Psi$ becomes a frame. In this paper, we show that if $E$ acts surjectively on $\bigoplus_{n=1}^{\infty}\mathcal{H}$, then every ordinary frame is an $E$-frame.
\end{abstract}
\section{Introduction and Preliminaries}
Frames were originally introduced by Duffin and Schaeffer in 1952 \cite{DS} and were later revisited by Daubechies, Grossmann, and Meyer in 1986 \cite{DG}, after which their applications became widespread. Due to their advantageous properties, frames have been extensively used in the characterization of function spaces, signal processing, and various other areas. We refer the reader to \cite{ch,CA,HAN} for an introduction to frame theory and its applications. Over the past decades, several generalizations of frames have been developed for specific applications, including frames of subspaces, fusion frames, and $g$-frames.

Recently, matrix-generated frame structures have attracted attention. In particular, the notion of an $E$-frame, introduced in \cite{TAL}, extends the classical frame concept by applying an infinite matrix transformation to a sequence in a Hilbert space. More precisely, given an infinite matrix $E=(E_{n,k})_{n,k\geq 1}$, a sequence $\Psi=\{\psi_k\}_{k=1}^{\infty}$ is transformed into the sequence
$E\Psi=\left\{\sum_{k=1}^{\infty}E_{n,k}\psi_k\right\}_{n=1}^{\infty}$,
whenever the corresponding series are convergent. This approach provides a flexible framework for studying frame-like properties under infinite matrix transformations.

The main purpose of this paper is to investigate the relationship between infinite matrix mappings and $E$-frames in separable Hilbert spaces. Since infinite matrices do not necessarily define bounded operators on Hilbert sequence spaces, establishing suitable conditions under which they induce bounded operators is an essential step. We first study infinite matrices acting on the Hilbert space $\bigoplus_{n=1}^{\infty}\mathcal{H}$ 
and prove that, under appropriate assumptions, such matrices induce bounded operators on this space. We also show that these conditions imply the boundedness of the corresponding operators on $\ell^2(\mathbb{N})$.

Using these operator-theoretic results, we prove that if an infinite invertible matrix defines a surjective matrix mapping on $\bigoplus_{n=1}^{\infty}\mathcal{H}$, then it induces bounded bijective operators on both $\bigoplus_{n=1}^{\infty}\mathcal{H}$ and $\ell^2(\mathbb{N})$. These properties allow us to establish that every classical frame satisfying suitable assumptions becomes an $E$-frame. In particular, we obtain new classes of $E$-frames generated by infinite invertible matrix transformations.

We also investigate the behavior of matrix conjugation and show that the conjugate matrix associated with a suitable infinite matrix preserves the relevant mapping properties. Finally, we characterize when a Bessel sequence becomes an $E$-Bessel sequence in terms of the associated matrix transformation on the analysis coefficients.

The paper is organized as follows. First, we recall the necessary definitions and preliminary results concerning frames, Bessel sequences, and matrix mappings. In Section 2, we establish the main operator-theoretic results for infinite matrices and apply them to the theory of $E$-frames.

Throughout this paper, we assume that $\mathcal H$ is a separable Hilbert space. A countable family $\mathcal{F}=\lbrace f_{k}\rbrace_{k=1}^{\infty}$ in  $\mathcal{H}$ is called a frame for $\mathcal{H}$ if there exist constants $0<A\leq B<\infty$ such that 
\begin{equation}
A\Vert f\Vert^{2}\leq\sum_{k=1}^{\infty}\big\vert\big<f,f_{k}\big>\big\vert^{2}\leq B\Vert f\Vert^{2}\qquad (f\in\mathcal{H}).
\end{equation}
$A$ and $B$ are called the frame bounds. If just the right inequality holds, we say $\mathcal{F}$ is a Bessel sequence for $\mathcal{H}$ with bound $B$. A Riesz basis for $\mathcal{H}$ is a family of the form $\lbrace Ue_{k}\rbrace_{k=1}^{\infty}$, where $\lbrace e_{k}\rbrace_{k=1}^{\infty}$ is an orthonormal basis for $\mathcal{H}$ and $U:\mathcal{H}\longrightarrow\mathcal{H}$ is a bounded bijective operator. \\ 
For the sequence $(\mathcal{H}_{n})_{n=1}^{\infty}$ of separable Hilbert spaces, suppose that $$\bigoplus_{n=1}^{\infty}\mathcal{H}_{n}=\bigg\lbrace\lbrace f_{n}\rbrace_{n=1}^{\infty}| f_{n}\in\mathcal{H}_{n},\sum_{n=1}^{\infty}\|f_{n}\|^{2}<\infty\bigg\rbrace.$$ We can define a well defined inner product $\big<.\,,.\big>$ on $\bigoplus_{n=1}^{\infty}\mathcal{H}_{n}$ defined by $$\big<\lbrace f_{n}\rbrace_{n=1}^{\infty},\lbrace g_{n}\rbrace_{n=1}^{\infty}\big>=\sum_{n=1}^{\infty}\big<f_{n},g_{n}\big>.$$ It is well known that $\bigoplus_{n=1}^{\infty}\mathcal{H}_{n}$ is a Hilbert space with respect to this inner product which is called the Hilbert space direct sum of $(\mathcal{H}_{n})_{n=1}^{\infty}$ \cite{Conway}.\\
\begin{definition}
Let $\mathcal X$ and $\mathcal Y$ be two sequence spaces and $E=(E_{n,k})_{n,k\geq 1}$ be an infinite matrix of real or complex numbers. We say that $E$ defines a matrix mapping from $\mathcal X$ into $\mathcal Y$, if for every sequence $x=\{x_n\}_{n=1}^{\infty}$ in $\mathcal X$, the sequnce $Ex=\{(Ex)_n\}_{n=1}^{\infty}$ is in $\mathcal Y$, where
$$(Ex)_n=\sum_{k=1}^{\infty}E_{n,k}x_k,\quad\left(n\in\mathbb{N}\right).$$
\end{definition}
Using the matrix mapping concept,  the authors in \cite{TAL}, have introduced a new notion of frames which is called $E$-frame. Let $E$ be an infinite invertible matrix mapping on $\bigoplus_{n=1}^{\infty}\mathcal{H}$. Then for each $\mathcal{F}=\lbrace f_{k}\rbrace_{k=1}^{\infty}\in\bigoplus_{n=1}^{\infty}\mathcal{H}$, $$E\mathcal{F}=\left\lbrace\sum_{k=1}^{\infty}E_{n,k}f_{k}\right\rbrace_{n=1}^{\infty}.$$
\begin{definition}\cite{TAL}
The sequence $\mathcal{F}=
\lbrace f_{k}\rbrace_{k=1}^{\infty}$ is called an $E$-frame for $\mathcal{H}$ if there exist constants $0<A\leq B<\infty$ such that 
\begin{equation}\label{$E$-frame}
A\|f\|^{2}\leq\sum_{n=1}^{\infty}\left\vert\left<f,\left(E\mathcal{F}\right)_{n}\right>\right\vert^{2}\leq B\|f\|^{2}\qquad (f\in\mathcal{H}).
\end{equation}
\end{definition}
If just the right inequality in (\ref{$E$-frame}) holds, then $\mathcal{F}$ is called an $E$-Bessel sequence for $\mathcal{H}$ with  $E$-Bessel bound $B$. One can easily check that for given constant $B>0$, the sequence $\mathcal{F}$ is an $E$-Bessel sequence if and only if the operator $T_{E\mathcal{F}}$ defined by $$T_{E\mathcal{F}}:\ell^{2}(\mathbb{N}):\longrightarrow\mathcal{H}\,,T\lbrace c_{k}\rbrace_{k=1}^{\infty}=\sum_{k=1}^{\infty}c_{k}\left(E\mathcal{F}\right)_{k}$$ is a bounded operator from $\ell^{2}(\mathbb{N})$ in to $\mathcal{H}$ with $\|T\|\leq\sqrt{B}$. We call $T_{E\mathcal{F}}$ the pre $E$-frame operator. It's adjoint, the analysis operator, is given by 
\begin{equation}\label{analysis}
T_{E\mathcal{F}}^{*}:\mathcal{H}\longrightarrow\ell^{2}(\mathbb{N})\,,T_{E\mathcal{F}}^{*}f=\left\lbrace\left<f,\left(E\mathcal{F}\right)_{k}\right>\right\rbrace_{k=1}^{\infty}.
\end{equation} Composing $T_{E\mathcal{F}}$ and $T_{E\mathcal{F}}^{*}$, the $E$-frame operator $$S_{E\mathcal{F}}:\mathcal{H}\longrightarrow\mathcal{H}\,,S_{E\mathcal{F}}f=\sum_{k=1}^{\infty}\left<f,\left(E\mathcal{F}\right)_{k}\right>\left(E\mathcal{F}\right)_{k}$$ is obtained. It is easily could be checked that $S_{E\mathcal{F}}$ is bounded, invertible, self-adjoint and positive \cite{TAL}. This leads us to the following reconstruction formulas
\begin{equation}\label{rec1}
f=\sum_{n=1}^{\infty}\left<f,\left(E\lbrace S_{E\mathcal{F}}^{-1}f_{k}\rbrace_{k=1}^{\infty}\right)_{n}\right>\left(E\mathcal{F}\right)_{n}=\sum_{n=1}^{\infty}\left<f,S_{E\mathcal{F}}^{-1}\left(E\mathcal{F}\right)_{n}\right>\left(E\mathcal{F}\right)_{n}
\end{equation}
 and $$f=\sum_{n=1}^{\infty}\left<f,\left(E\mathcal{F}\right)_{n}\right>\left(E\lbrace S_{E\mathcal{F}}^{-1}f_{k}\rbrace_{k=1}^{\infty}\right)_{n}=\sum_{n=1}^{\infty}\left<f,\left(E\mathcal{F}\right)_{n}\right>S_{E\mathcal{F}}^{-1}\left(E\mathcal{F}\right)_{n},$$ for all $f\in\mathcal{H}$.
Let $\lbrace e_{k}\rbrace_{k=1}^{\infty}$ be an orthonormal basis for a separable Hilbert space $\mathcal{H}$. An $E$-Riesz basis for $\mathcal{H}$ is a family of the form $\big\lbrace U\big(E^{-1}\lbrace e_{k}\rbrace_{k=1}^{\infty}\big)_{n}\big\rbrace_{k=1}^{\infty}$, where $U$ is a bounded bijective operator on $\mathcal{H}$.
\section{main results}\label{main}
In this section, we suppose that $\mathcal{H}$ is a separable Hilbert space and $E=(E_{n,k})_{n,k\geq 1}$ is an infinite complex or real matrix. 

\begin{definition}
We say a sequence $\Psi=\lbrace \psi_k\rbrace_{k=1}^{\infty}$ is an $E$-sequence in $\mathcal{H}$ if, its $E$-transform, i.e. the sequence 
\begin{equation}\label{eq8}
E\Psi=\left\lbrace\sum_{k=1}^{\infty}E_{n,k}\psi_k\right\rbrace_{n=1}^{\infty},
\end{equation}
is a well defined sequence in $\mathcal{H}$, that is, the series in \eqref{eq8} converges in $\mathcal{H}$ for all
 $n\in\mathbb{N}$.
\end{definition}
\begin{lemma}\label{lem1}
Suppose that $E$ is invertible and $\Psi=\lbrace \psi_k\rbrace_{k=1}^{\infty}$ is a sequence in $\mathcal{H}$. Then
\begin{itemize}
\item[$(i)$] If $\Psi$ is an $E$-sequence then $E\Psi$ is an $E^{-1}$-sequence.
\item[$(ii)$] If $\Psi$ is an $E^{-1}$-sequence then $E^{-1}\Psi$ is an $E$-sequence.
\end{itemize}
\end{lemma}
\begin{proof}
$(i)$ Let $\Psi$ be an $E$-sequence. Then for each $n\in\mathbb{N}$
	\begin{align}\label{eq9}
		\left(E^{-1}\left\{\left(E\Psi\right)_j\right\}_{j=1}^\infty\right)_n&=\sum_{j=1}^\infty E^{-1}_{n,j}\left(E\Psi\right)_j\\&=\sum_{j=1}^\infty E^{-1}_{n,j}\sum_{k=1}^\infty E_{j,k}\psi_k=\sum_{j=1}^\infty\sum_{k=1}^\infty E^{-1}_{n,j}E_{j,k}\psi_k\nonumber.
	\end{align}
Using this fact that $E^{-1}$ is the algebraic inverse of $E$, for fixed $k\in\mathbb{N}$ we have
\begin{equation}
\sum_{j=1}^\infty E^{-1}_{n,j}E_{j,k}\psi_k=\begin{cases}
		0&n\neq k,\\\psi_k&n=k.\end{cases}
\end{equation}
Hence, \eqref{eq9} implies that $\left(E^{-1}\left\{\left(E\Psi\right)_j\right\}_{j=1}^\infty\right)_n=\psi_n$.
The proof of $(ii)$ is similar.
\end{proof}
\begin{lemma}\label{lem2}
Assume that $\mathcal{H}$ is a nonzero Hilbert space. If $E$ is a matrix mapping on $\bigoplus_{n=1}^\infty\mathcal{H}$, then it is a matrix mapping on $\ell^2(\mathbb{N})$.
\end{lemma}
\begin{proof}
If $\mathcal{H}=\mathbb{C}$ the result is trivial. Let $\mathcal{H}$ be any other nonzero Hilbert space and 
suppose that $\left\{c_k\right\}_{k=1}^\infty$ is an arbitrary sequence in $\ell^2(\mathbb{N})$. Let $f$ be a nonzero vector in $\mathcal{H}$. define a sequence $\Psi=\left\{\psi_k\right\}_{k=1}^\infty$ where $\psi_k=c_kf$ for all $k\in\mathbb{N}$. $\Psi\in\bigoplus_{n=1}^\infty\mathcal{H}$ since,
\begin{align*}
\sum_{k=1}^\infty\Vert\psi_k\Vert^2=\sum_{k=1}^\infty\Vert c_kf\Vert^2=\sum_{k=1}^\infty\vert c_k\vert^2\Vert f\Vert^2=\Vert f\Vert^2\sum_{k=1}^\infty\vert c_k\vert^2<\infty.
\end{align*}
Hence $E\Psi\in\bigoplus_{n=1}^\infty\mathcal{H}$ by assumption, that is, $\sum_{n=1}^\infty\left\Vert\left(E\Psi\right)_n\right\Vert^2<\infty.$ So we have
\begin{align}\label{eq8.1}
\sum_{n=1}^\infty\left\Vert\left(E\Psi\right)_n\right\Vert^2=\sum_{n=1}^\infty\left\Vert\sum_{k=1}^\infty E_{n,k}\psi_k\right\Vert^2=\Vert f\Vert^2\sum_{n=1}^\infty\left\vert\sum_{k=1}^\infty E_{n,k}c_k\right\vert^2<\infty.
\end{align}
Since $f$ is a nonzero vector, \eqref{eq8.1} implies that $\sum_{n=1}^\infty\left\vert\sum_{k=1}^\infty E_{n,k}c_k\right\vert^2<\infty.$ This exactly means that $E:\ell^2(\mathbb{N})\to \ell^2(\mathbb{N})$ is well-defined.
\end{proof}
\begin{theorem}\label{th1}
Let $E$ defines a matrix mapping on $\bigoplus_{n=1}^\infty\mathcal{H}$. Then $E$ induces a bounded operator on the Hilbert space $\bigoplus_{n=1}^\infty\mathcal{H}$.
\end{theorem}
\begin{proof}
Consider the following mapping on $\bigoplus_{n=1}^\infty\mathcal{H}$.
\begin{equation}\label{eqT}
T:\bigoplus_{n=1}^\infty\mathcal{H}\longrightarrow\bigoplus_{n=1}^\infty\mathcal{H}\qquad;\qquad T\Psi=\left\{\sum_{k=1}^\infty E_{n,k}\psi_k\right\}_{n=1}^\infty,
\end{equation}
for all $\Psi=\left\{\psi_k\right\}_{k=1}^\infty\in\bigoplus_{n=1}^\infty\mathcal{H}$.
The hypothesis says that $T\Psi\in\bigoplus_{n=1}^\infty\mathcal{H}$ for all $\psi\in\bigoplus_{n=1}^\infty\mathcal{H}$. So $T$ is well defined. For a fixed $n\in\mathbb{N}$, suppose that $T_n:\bigoplus_{n=1}^\infty\mathcal{H}\longrightarrow\mathcal{H}$ is the $n$th coordinate functional of $T$ defined by $T_n\Psi=(T\Psi)_n$.

\textbf{Step1}: First we show that $T_n$ is linear and bounded for all $n\in\mathbb{N}$. For this purpose, we use partial sums of series in \eqref{eqT} to define the linear operators
\begin{equation}\label{eqTn}
T_n^{(m)}:\bigoplus_{n=1}^\infty\mathcal{H}\longrightarrow\mathcal{H}\qquad;\qquad T_n^{(m)}\Psi=\sum_{k=1}^m E_{n,k}\psi_k.
\end{equation}
we claim that $T_n^{(m)}$ is bounded for all $m\in\mathbb{N}$. In fact, for all $\Psi\in\bigoplus_{n=1}^\infty\mathcal{H}$,
\begin{align*}
\left\Vert T_n^{(m)}\Psi\right\Vert_{\mathcal{H}}^2=\left\Vert\sum_{k=1}^m E_{n,k}\psi_k\right\Vert_{\mathcal{H}}^2&\leq\left(\sum_{k=1}^m\left\vert E_{n,k}\right\vert\left\Vert\psi_k\right\Vert\right)^2\\&\leq\sum_{k=1}^m \left\vert E_{n,k}\right\vert^2\sum_{k=1}^m\left\Vert\psi_k\right\Vert^2\leq\sum_{k=1}^m \left\vert E_{n,k}\right\vert^2\left\Vert\Psi\right\Vert_{\bigoplus\mathcal{H}}^2, 
\end{align*}
and therefore $\left\Vert T_n^{(m)}\right\Vert\leq\left(\sum_{k=1}^m \left\vert E_{n,k}\right\vert^2\right)^{\frac{1}{2}}<\infty.$

Moreover, for fixed $\Psi\in\bigoplus_{n=1}^\infty\mathcal{H}$, 
\begin{equation*}
\lim_{m\to\infty}T_n^m\Psi=\lim_{m\to\infty}\sum_{k=1}^m E_{n,k}\psi_k=\sum_{k=1}^\infty E_{n,k}\psi_k=T_n\Psi.
\end{equation*}
Therefore $\left\{T_n^m\right\}_{m=1}^\infty$ is a sequence of bounded and linear operators between Banach spaces $\bigoplus_{n=1}^\infty\mathcal{H}$ and $\mathcal{H}$ which converges pointwise to the map $T_n:\bigoplus_{n=1}^\infty\mathcal{H}\longrightarrow\mathcal{H}$. In fact $\sup_m\|T_n^m\Psi\|<\infty$, hence the Banach–Steinhaus Theorem shows that $T_n$ is linear and bounded. This also leads to linearity of $T$.

\textbf{Step2:} Now we are going to show that $T$ is bounded. Assume that $\left\{\Psi^l\right\}_{l=1}^\infty$ is a sequence in $\bigoplus_{n=1}^\infty\mathcal{H}$ such that $\Psi^l\xrightarrow{\Vert.\Vert_{\bigoplus\mathcal{H}}}\Psi$ and $T\Psi^l\xrightarrow{\Vert.\Vert_{\bigoplus\mathcal{H}}}\Phi$ for some $\Phi=\left\{\phi_k\right\}_{k=1}^\infty\in\bigoplus_{n=1}^\infty\mathcal{H}$.Since both spaces are Banach spaces, it suffices to show that the graph of $T$ is closed. We show this is true. Indeed, since $T\Psi^l\xrightarrow{\Vert.\Vert_{\bigoplus\mathcal{H}}}\Phi$, hence
\begin{align}\label{eq3}
\sum_{n=1}^\infty\left\Vert\left(T\Psi^l\right)_n-\phi_n\right\Vert_{\mathcal{H}}^2=\left\Vert T\Psi^l-\Phi\right\Vert_{\bigoplus\mathcal{H}}^2\longrightarrow0,
\end{align} 
as $l$ tends to infinity. For fixed $n\in\mathbb{N}$, as $\left\Vert\left(T\Psi^l\right)_n-\phi_n\right\Vert_{\mathcal{H}}^2\leq\left\Vert T\Psi^l-\Phi\right\Vert_{\bigoplus\mathcal{H}}^2$, \eqref{eq3} implies that 
\begin{align}\label{eq5}
lim_{l\to\infty}\left(T\Psi^l\right)_n=\phi_n.
\end{align}
On the other hand, By first part of the proof, $T_n$ is continuous. Hence, because $\Psi^l\xrightarrow{\Vert.\Vert_{\bigoplus\mathcal{H}}}\Psi$, we have 
\begin{align}\label{eq6}
lim_{\l\to\infty}T_n\Psi^l=T_n\Psi.
\end{align}
But $T_n\Psi^l=\left(T\Psi^l\right)_n$ for each $l\in\mathbb{N}$, by definition. Now, by \eqref{eq5} and using the uniqueness of limits in the Banach space $\mathcal{H}$, we obtain
\begin{align}\label{eq7}
\left(T\Psi\right)_n=T_n\Psi=\lim_{l\to\infty}T_n\Psi^l=\lim_{l\to\infty}\left(T\Psi^l\right)_n=\phi_n.
\end{align}
Since we had fixed $n\in\mathbb{N}$, \eqref{eq7} implies that $T\Psi=\Phi$ and the proof is complete.
\end{proof}
\begin{theorem}\label{th5}
Let $E$ be a matrix mapping on $\ell^2(\mathbb{N})$. Then $E$ induces a bounded operator on $\ell^2(\mathbb{N})$.
\end{theorem}
\begin{proof}
Suppose $U:\ell^2(\mathbb{N})\longrightarrow\ell^2(\mathbb{N})$ is a mapping which is defined by
\begin{align*}
Uc=\left\{\sum_{k=1}^\infty E_{n,k}c_k\right\}_{n=1}^\infty\qquad\left(c=\left\{c_k\right\}_{k=1}^\infty\in\ell^2(\mathbb{N})\right).
\end{align*}
Since $E$ is a matrix mapping on $\ell^2(\mathbb{N})$, hence $U$ is well defined. To reach our desired result, namely the boundedness and the linearity of $U$, we first show that for each $n\in\mathbb{N}$, the coordinate functional $U_n:\ell^2(\mathbb{N})\longrightarrow\mathbb{C}$, defined by $U_nc=\left(Uc\right)_n$ with $c\in\ell^2(\mathbb{N})$, is linear and bounded. To this purpose, fix $n\in\mathbb{N}$ and for each $m\in\mathbb{N}$ define the linear operator $U_n^m:\ell^2(\mathbb{N})\longrightarrow\mathbb{C}$ by 
\begin{align*}
U_n^mc=\sum_{k=1}^m E_{n,k}c_k\qquad\left(c=\left\{c_k\right\}_{k=1}^\infty\in\ell^2(\mathbb{N})\right).
\end{align*}
For each $m\in\mathbb{N}$, $U_n^m$ is bounded. In fact, for each $c=\left\{c_k\right\}_{k=1}\in\ell^2(\mathbb{N})$,
\begin{align*}
\left|U_n^mc\right|^2=\left|\sum_{k=1}^mE_{n,k}c_k\right|^2\leq\sum_{k=1}^m\left|E_{n,k}\right|^2\sum_{k=1}^m\left|c_k\right|^2\leq\sum_{k=1}^m\left|E_{n,k}\right|^2\left\Vert c\right\Vert_{\ell^2}^2.
\end{align*}
Thus $\left\Vert U_n^m\right\Vert\leq\left(\sum_{k=1}^m\left|E_{n,k}\right|^2\right)^{\frac{1}{2}}<\infty.$

On the other hand, 
\begin{align*}
\lim_{m\to\infty}U_n^mc=\lim_{m\to\infty}\sum_{k=1}^mE_{n,k}c_k=\sum_{k=1}^\infty E_{n,k}c_k=U_nc.
\end{align*}
In this way, we showed that $\left\{U_n^m\right\}_{m=1}^\infty$ is a sequence of bounded and linear functionals on $\ell^2(\mathbb{N})$ which converges pointwise to the mapping $U_n$. Now we apply the Banach-Steinhaus Theorem to prove that $U_n$ is linear and bounded and so the linearity of $U$ is obtained.

As the last part of the proof, assume that $\left\{c^i\right\}_{l=1}^\infty$ is a sequence in $\ell^2(\mathbb{N})$ such that $c^i\xrightarrow{\Vert.\Vert_{\ell^2}}c$ and $Uc^i\xrightarrow{\Vert.\Vert_{\ell^2}}d$, for some $c,d\in\ell^2(\mathbb{N})$. We show $Uc=d$. Indeed, for fixed $n\in\mathbb{N}$ we have
\begin{align*}
\left|\left(Uc^i\right)_n-d_n\right|^2\leq\sum_{k=1}^\infty\left|\left(Uc^i\right)_k-d_k\right|^2=\left\Vert Uc^i-d\right\Vert_{\ell^2}^2\longrightarrow0.
\end{align*}
Hence
\begin{align}\label{eq11}
\lim_{i\to\infty}\left(Uc^i\right)_n=d_n.
\end{align}
On the other hand, we have showed that $U_n$ is a bounded linear operator on $\ell^2(\mathbb{N})$ for each $n\in\mathbb{N}$. This leads to
\begin{align}\label{eq12}
\lim_{i\to\infty}U_nc^i=U_n\left(\lim_{i\to\infty}c^i\right)=U_nc.
\end{align}
But $U_nc=\left(Uc\right)_n$ by definition. Therefore \eqref{eq11} and \eqref{eq12} imply that
\begin{align*}
d_n=\lim_{i\to\infty}\left(Uc^i\right)_n=\lim_{i\to\infty}U_nc^i=U_nc=(Uc)_n,
\end{align*}
and this means that $Uc=d$. Now an application of the Closed Graph Theorem proves that $U$ is bounded.
\end{proof}
\begin{corollary}\label{col1}
Let $E$ be a matrix with algebraic inverse $E^{-1}$. Suppose that $E$ defines a matrix mapping on $\bigoplus_{n=1}^\infty\mathcal{H}$. If $E$ acts on $\bigoplus_{n=1}^\infty\mathcal{H}$ surjectively, then it induces a bounded bijection T on $\bigoplus_{n=1}^\infty\mathcal{H}$, where $T^{-1}$ is bounded and represents by $E^{-1}$.
\end{corollary}
\begin{proof}
By Theorem \ref{th1}, $E$ induces a bounded linear operator $T$ on $\bigoplus_{n=1}^\infty\mathcal{H}$. We show $T$ is surjective and injective. 

The surjectivity of $T$ is straightly concluded by assumption, that is, for each $\Psi\in\bigoplus_{n=1}^\infty\mathcal{H}$, there exists $\Phi\in\bigoplus_{n=1}^\infty\mathcal{H}$ such that $E\Phi=\Psi$ which is equivalent to $T\Phi=\Psi$.

Now suppose that $T\Psi=0$ for some $\Psi\in\bigoplus_{n=1}^\infty\mathcal{H}$. We use Lemma \ref{lem1} to obtain
\begin{align}
\Psi=E^{-1}\left(E\Psi\right)=E^{-1}\left(T\Psi\right)=E^{-1}(0)=0,
\end{align}
and $T$ is therefore bijective. That $T^{-1}$ is bounded is concluded by applying the Open Mapping Theorem.

Finally, we show that for all $\Psi\in\bigoplus_{n=1}^\infty\mathcal{H}$, $T^{-1}\Psi=E^{-1}\Psi$. Indeed, for any $\Psi\in\bigoplus_{n=1}^\infty\mathcal{H}$ there is $\Phi\in\bigoplus_{n=1}^\infty\mathcal{H}$ such that $T\Phi=E\Phi=\Psi$. Hence, by Lemma \ref{lem1}
\begin{equation}
\Phi=E^{-1}\left(E\Phi\right)=E^{-1}\Psi.
\end{equation}
But $\Phi$ is exactly $T^{-1}\Psi$ and therefore $T^{-1}\Psi=E^{-1}\Psi$.
\end{proof}
\begin{theorem}\label{th4}
Let $E$ be an invertible matrix where defines a surjective matrix mapping on $\bigoplus_{n=1}^{\infty}\mathcal{H}$. Then $E$ induces a bounded bijection $U$ on $\ell^2(\mathbb{N})$ where $U^{-1}$ represents by $E^{-1}$. 
\end{theorem}
\begin{proof}
Lemma \ref{lem2} and Theorem \ref{th5} show that $E$ induces a bounded operator $U$ on $\ell^2(\mathbb{N})$ defined by
\begin{align*}
U\{c_k\}_{k=1}^\infty=\left\{\sum_{k=1}^\infty E_{n,k}c_k\right\}_{n=1}^\infty,\qquad\left(\{c_k\}_{k=1}^\infty\in\ell^2(\mathbb{N})\right).
\end{align*}
First we prove that $U$ is surjective on $\ell^2(\mathbb{N})$. To this end, let $c=\{c_k\}_{k=1}^\infty$ be a given sequence in $\ell^2(\mathbb{N})$. Without losing of generality, suppose that $\mathcal{H}$ is a nonzero Hilbert space and take a nonzero element $f\in\mathcal{H}$. Now define a sequence $\Psi=\{\psi_k\}_{k=1}^\infty$ in $\mathcal{H}$, by $\psi_k=c_kf(k\in\mathbb{N})$. We show that $\Psi\in\bigoplus_{n=1}^\infty\mathcal{H}$. In fact,
\begin{align*}
\sum_{k=1}^\infty\left\Vert\psi_k\right\Vert^2=\sum_{k=1}^\infty\left\Vert c_kf\right\Vert^2=\Vert f\Vert^2\sum_{k=1}^\infty\left\vert c_k\right\vert^2<\infty.
\end{align*}
By Corollary \ref{col1}, $E$ induces a bounded bijection $T$ on $\bigoplus_{n=1}^\infty\mathcal{H}$. Thus there exists some sequence $\Phi=\{\phi_k\}_{k=1}^\infty\in\bigoplus_{n=1}^\infty\mathcal{H}$ such that $T\Phi=\Psi$ or equivalently
\begin{align*}
\left(T\Phi\right)_n=\sum_{k=1}^\infty E_{n,k}\phi_k=\psi_n=c_nf,\quad(n\in\mathbb{N}).
\end{align*}
Now we define a scalar sequence $d=\{d_k\}_{k=1}^\infty$ by
\begin{align*}
d_k=\frac{\left\langle\phi_k,f\right\rangle}{\Vert f\Vert^2},\quad\left(k\in\mathbb{N}\right),
\end{align*}
where $\langle\;.\;\rangle$ is the inner product of $\mathcal{H}$ and $\Vert f\Vert\neq0$ by our assumption. We show $d\in\ell^2(\mathbb{N})$. Indeed,
\begin{align*}
\sum_{k=1}^\infty\left\vert d_k\right\vert^2=\sum_{k=1}^\infty\frac{\left\vert\left\langle\phi_k,f\right\rangle\right\vert^2}{\Vert f\Vert^4}=\frac{1}{\Vert f\Vert^4}\sum_{k=1}^\infty\left\vert\left\langle\phi_k,f\right\rangle\right\vert^2\leq\frac{1}{\Vert f\Vert^2}\sum_{k=1}^\infty\left\Vert\phi_k\right\Vert^2<\infty.
\end{align*}
As the final step, we show that $Ud=c$. In fact, for each $n\in\mathbb{N}$, 
\begin{align}\label{eq9.1}
(Ud)_n=\sum_{k=1}^\infty E_{n,k}d_k=\sum_{k=1}^\infty E_{n,k}\frac{\left\langle\phi_k,f\right\rangle}{\Vert f\Vert^2}.
\end{align}
Since $\sum_{k=1}^\infty E_{n,k}\phi_k<\infty$ by hypothesis and the inner product of $\mathcal{H}$ is continuous, \eqref{eq9.1} imply that 
\begin{align*}
(Ud)_n=\frac{\left\langle\sum_{k=1}^\infty E_{n,k}\phi_k,f\right\rangle}{\Vert f\Vert^2}=\frac{\left\langle c_nf,f\right\rangle}{\Vert f\Vert^2}=c_n\frac{\Vert f\Vert^2}{\Vert f\Vert^2}=c_n.
\end{align*}
This shows that $Ud=c$ and $U$ is therefore surjective. To prove the injectivity, suppose that $c$ is a sequence in $\ell^2(\mathbb{N})$ such that $Uc=0$. Then by definition of $U$ and Lemma \ref{lem1}
\begin{align*}
c=E^{-1}\left(Ec\right)=E^{-1}\left(Uc\right)=E^{-1}(0)=0.
\end{align*}
This shows that $U$ is a bounded linear bijection on the Banach space $\ell^2(\mathbb{N})$, and by an application of Open Mapping Theorem, $U^{-1}$ is bounded. Moreover for given sequence $c\in\ell^2(\mathbb{N})$, there exists $d\in\ell^2(\mathbb{N})$ such that $Ud=c$ or equivalently $d=U^{-1}c$. On the other hand $U^{-1}c=d=E^{-1}\left(Ed\right)=E^{-1}(Ud)=E^{-1}c$. Thus $U^{-1}c=E^{-1}c$ for all $c\in\ell^2(\mathbb{N})$ and the proof is complete.
\end{proof}
As a preparation for Theorem \ref{th3}, we state the following lemma
\begin{lemma}\label{lem3}
Let $E$ be an infinite matrix which defines a bounded operator on $\bigoplus_{n=1}^\infty\mathcal{H}$. Then the rows and the columns of $E$ belongs to $\ell^2(\mathbb{N})$.
\end{lemma}
\begin{proof}
Suppose that $T:\ell^2(\mathbb{N})\longrightarrow\ell^2(\mathbb{N})$ is the bounded linear operator defined by
\begin{align}\label{eq10}
Tc=\left\{\sum_{k=1}^\infty E_{n,k}c_k\right\}_{n=1}^\infty,
\end{align}
for each $c=\{c_k\}_{k=1}^\infty$. For fixed $n\in\mathbb{N}$, the mapping $T_n:\ell^2(\mathbb{N})\longrightarrow\mathbb{C}$ defined by $T_nc=(Tc)_n$, is a bounded linear functional. Indeed, for each $c=\{c_k\}_{k=1}^\infty$,
\begin{align*}
\left\vert T_nc\right\vert^2=\left\vert\left(Tc\right)_n\right\vert^2\leq\sum_{n=1}^\infty\left\vert\left(Tc\right)_n\right\vert^2=\left\Vert\left\{\left(Tc\right)_n\right\}_{n=1}^\infty\right\Vert_{\ell^2}^2=\left\Vert Tc\right\Vert_{\ell^2}^2\leq\left\Vert T\right\Vert^2\left\Vert c\right\Vert_{\ell^2}^2.
\end{align*}
Hence by Riesz Representation Theorem, there exist a unique sequence $d=\{d_k\}_{k=1}^\infty\in\ell^2(\mathbb{N})$ such that 
\begin{align*}
T_nc=\left\langle c,d\right\rangle_{\ell^2}=\sum_{k=1}^\infty c_k\overline{d_k}\qquad\left(c\in\ell^2(\mathbb{N})\right).
\end{align*}
But by \eqref{eq10}, 
\begin{align*}
T_nc=\sum_{k=1}^\infty E_{n,k}c_k\qquad\left(c\in\ell^2(\mathbb{N})\right). 
\end{align*}
This shows that $\left\{E_{n,k}\right\}_{k=1}^\infty=\left\{\overline{d_k}\right\}_{k=1}^\infty$ and so $\left\{E_{n,k}\right\}_{k=1}^\infty\in\ell^2(\mathbb{N})$.
To prove the remain, just note that the adjoint matrix $E^\ast$ represents the adjoint operator $T^\ast$, which is bounded of course.
\end{proof}
Motivated by Theorem \ref{th4}, we prove our main result which asserts that a classical frame can be considered as an $E$-frame.
\begin{theorem}\label{th3}
Suppose that $E=(E_{n,k})_{n,k\geq 1}$ is an invertible matrix which defines a surjective matrix mapping on $\bigoplus_{n=1}^\infty\mathcal{H}$. Suppose that $\Psi=\lbrace\psi_{k}\rbrace_{k=1}^{\infty}$ is a frame for $\mathcal{H}$ with bounds $A$ and $B$. Then $\Psi$ is an $E$-frame with bounds $\Vert E\Vert^{2}B$ and $CA$, for some $C>0$. 
\end{theorem}
\begin{proof}
By Theorem \ref{th4}, the matrix $E$ induces a bounded linear bijection on $\ell^2(\mathbb{N})$, we note it by $E$ again. Lemma \ref{lem3} states that $\left\{E_{n,k}\right\}_{k=1}^\infty\in\ell^2(\mathbb{N})$ for all $n\in\mathbb{N}$. This with \cite[Corollary 3.2.5]{ch} makes $\sum_{k=1}^\infty E_{n,k}\psi_k$ converge for all $n\in\mathbb{N}$. Hence $\Psi$ is an $E$-sequence. Noting that $\left\lbrace\left\langle f,\psi_{k}\right\rangle\right\rbrace_{k=1}^{\infty}\in\ell^{2}(\mathbb{N})$ and using this fact that the inner product of $\mathcal{H}$ is continuous we have
\begin{align}\label{eq1}
\sum_{n=1}^{\infty}\left\vert\left\langle f,\left(E\Psi\right)_{n}\right\rangle\right\vert^{2}&=\sum_{n=1}^{\infty}\left\vert\left\langle f,\sum_{k=1}^{\infty}E_{n,k}\psi_{k}\right\rangle\right\vert^{2}=\sum_{n=1}^{\infty}\left\vert\sum_{k=1}^{\infty}\overline{E_{n,k}}\left\langle f,\psi_{k}\right\rangle\right\vert^{2}\\&=\sum_{n=1}^{\infty}\left\vert\sum_{k=1}^{\infty}{E_{n,k}}\left\langle\psi_{k},f\right\rangle\right\vert^{2}=\left\Vert E\left\lbrace\left\langle\psi_{k},f\right\rangle\right\rbrace_{k=1}^{\infty}\right\Vert_{\ell^{2}}^{2}\nonumber\\&\leq\Vert E\Vert^{2}\sum_{k=1}^{\infty}\left\vert\left\langle \psi_{k},f\right\rangle\right\vert^{2}\leq\Vert E\Vert^{2}B\Vert f\Vert^{2}\nonumber.
\end{align}
To find a lower $E$-frame bound for $\Psi$ we note that since $E$ is invertible, so $E^{*}$ is a bijection. Since $E^{-1}$ is bounded 
$$\left\Vert\lbrace c_{k}\rbrace_{k=1}^{\infty}\right\Vert_{\ell^{2}}=\|E^{-1}E\lbrace c_{k}\rbrace_{k=1}^{\infty}\|_{\ell^{2}}\leq\|E^{-1}\|\left\Vert E\lbrace c_{k}\rbrace_{k=1}^{\infty}\right\Vert_{\ell^{2}},$$
for all $\lbrace c_{k}\rbrace_{k=1}^{\infty}\in\ell^{2}(\mathbb{N})$. Putting $C=\frac{1}{\|E^{-1}\|}$, the argument we stated in (\ref{eq1}) implies that
\begin{equation*}
\sum_{n=1}^{\infty}\left\vert\left\langle f,\left(E\Psi\right)_{n}\right\rangle\right\vert^{2}=\left\Vert E\left\lbrace\left\langle \Psi_{k},f\right\rangle\right\rbrace_{k=1}^{\infty}\right\Vert_{\ell^{2}}^{2}\geq C\sum_{k=1}^{\infty}\left\vert\left\langle f,\Psi_{k}\right\rangle\right\vert^{2}\geq CA\Vert f\Vert^{2}.
\end{equation*}
\end{proof}
\begin{corollary}
Suppose that $E=(E_{n,k})_{n,k\geq 1}$ is a diagonal matrix such that $0<\inf_n\left\vert\lambda_n\right\vert\leq\sup_n\left\vert\lambda_n\right\vert<\infty$ where $\lbrace\lambda_n\rbrace_{n=1}^{\infty}=\lbrace E_{n,n}\rbrace_{n\geq1}$. Then any frame $\Psi=\lbrace\psi_k\rbrace_{k=1}^{\infty}$ for $\mathcal{H}$ with bounds $A$ and $B$, is an $E$-frame for $\mathcal{H}$ with bounds $CA$ and $\lambda^2B$ where $\lambda=sup_n\left\vert\lambda_n\right\vert$ and $C$ is a positive number.
\end{corollary}
\begin{proof}
That $\Psi$ is an $E$-sequence is trivial.
First note that
\begin{equation*}
E=\begin{pmatrix}
\lambda_1&0&0&0&\cdots\\
0&\lambda_2&0&0&\cdots\\
0&0&\lambda_3&0&\cdots\\
0&0&0&\lambda_4&\cdots\\
\vdots&\vdots&\vdots&\vdots&\ddots
\end{pmatrix}
\end{equation*}
is invertible where $E^{-1}=(E_{n,k}^{-1})_{n,k\geq 1}$ is a diagonal matrix with $E_{n,n}^{-1}=\lambda_{n}^{-1}$ for each $n\in\mathbb{N}$.
Moreover $E$ defines a matrix mapping on $\bigoplus_{n=1}^\infty\mathcal{H}$, because for all $\Phi=\lbrace\phi_k\rbrace_{k=1}^{\infty}\in\bigoplus_{n=1}^\infty\mathcal{H}$,
\begin{align*}
\sum_{n=1}^\infty\left\Vert\left(E\Phi\right)_n\right\Vert^2=\sum_{n=1}^\infty\left\Vert\lambda_n\phi_n\right\Vert^2\leq\lambda^2\sum_{n=1}^\infty\left\Vert\phi_n\right\Vert^2<\infty.
\end{align*}
Now we show $E$ acts surjectively on $\bigoplus_{n=1}^\infty\mathcal{H}$. To this end, take any  $\Phi=\lbrace\phi_k\rbrace_{k=1}^{\infty}\in\bigoplus_{n=1}^\infty\mathcal{H}$ and define $F=\left\{f_n\right\}_{n=1}^\infty$ defined by 
\begin{align*}
f_n=\frac{\phi_n}{\lambda_n}\quad\left(n\in\mathbb{N}\right).
\end{align*}
We show $F\in\bigoplus_{n=1}^\infty\mathcal{H}$. Indeed, 
\begin{align*}
\sum_{n=1}^\infty\left\Vert f_n\right\Vert^2=\sum_{n=1}^\infty\frac{\left\Vert\phi_n\right\Vert^2}{\left\vert\lambda_n\right\vert^2}\leq\frac{1}{a^2}\sum_{n=1}^\infty\left\Vert\phi_n\right\Vert^2<\infty,
\end{align*}
where $a=inf_n\left\vert\lambda_n\right\vert>0$. 
 Also, $(EF)_n=\lambda_nf_n=\phi_n$ for all $n\in\mathbb{N}$, that is $EF=\Phi$. Therefore the invertible matrix $E$ defines a surjective matrix mapping on $\bigoplus_{n=1}^\infty\mathcal{H}$. Hence, $E$ induces a bounded bijection $E$ on $\ell^2(\mathbb{N})$ by Theorem \ref{th4}. Thus Theorem \ref{th3} implies that $\Psi$ is an $E$-frame with lower bound $CA$ with $C>0$ and upper bound $\Vert E\Vert^2B$. But $\Vert E\Vert\leq\lambda$ because \begin{align*}
 \left\Vert Ec\right\Vert_{\ell^2}^2=\sum_{n=1}^\infty\left\vert\left(Ec\right)_n\right\vert^2&=\sum_{n=1}^\infty\left\vert\sum_{k=1}^\infty E_{n,k}c_k\right\vert^2\\&=\sum_{n=1}^\infty\left\vert\lambda_nc_n\right\vert^2\\&\leq\sum_{n=1}^\infty\left\vert\lambda_n\right\vert^2\left\vert c_n\right\vert^2\leq\lambda^2\sum_{n=1}^\infty\left\vert c_n\right\vert^2=\lambda^2\left\Vert c\right\Vert_{\ell^2}^2.
\end{align*}
 This completes the proof.
\end{proof}
\begin{remark}\label{rem1}
In this part, we define a standard conjugate linear isometry in Hilbert space $\mathcal{H}$. Let $\left\{e_k\right\}_{k=1}^\infty$ be an ortonormal basis for $\mathcal{H}$ and define the mapping $f\mapsto\bar{f}$ from $\mathcal{H}$ to $\mathcal{H}$ where $\bar{f}=\sum_{k=1}^\infty\overline{\left\langle f,e_k\right\rangle}e_k$.

This is a conjugate linear isometry. In fact for any $\alpha\in\mathbb{C}$ and any $f,g\in\mathcal{H}$ we have
\begin{align*}
\overline{\alpha f+g}=\sum_{k=1}^\infty\overline{\left\langle\alpha f+g,e_k\right\rangle}e_k=\bar{\alpha}\sum_{k=1}^\infty\overline{\left\langle f,e_k\right\rangle}e_k+\sum_{k=1}^\infty\overline{\left\langle g,e_k\right\rangle}e_k=\bar{\alpha}\bar{f}+\bar{g}. 
\end{align*}
Also
\begin{align}\label{eq13}
\left\Vert\bar{f}\right\Vert^2=\left\langle\bar{f},\bar{f}\right\rangle&=\left\langle\sum_{k=1}^\infty\overline{\left\langle f,e_k\right\rangle}e_k,\sum_{j=1}^\infty\overline{\left\langle f,e_j\right\rangle}e_j\right\rangle\\&=\sum_{k=1}^\infty\overline{\left\langle f,e_k\right\rangle}\sum_{j=1}^\infty\left\langle f,e_j\right\rangle\left\langle e_k,e_j\right\rangle\nonumber\\&=\sum_{k=1}^\infty\overline{\left\langle f,e_k\right\rangle}\left\langle f,e_k\right\rangle\nonumber\\&=\sum_{k=1}^\infty\left|\left\langle f,e_k\right\rangle\right|^2=\left\Vert f\right\Vert^2\nonumber.
\end{align}
\eqref{eq13} indicates that the conjugate linear isometry is continuous on $\mathcal{H}$.
We can extend this conjugate linear isometry to $\bigoplus_{n=1}^\infty\mathcal{H}$ which for each $\Psi=\left\{\psi_k\right\}_{k=1}^\infty\in\bigoplus_{n=1}^\infty\mathcal{H}$ is defined by
\begin{align*}
\bar{\Psi}=\left\{\bar{\psi_k}\right\}_{k=1}^\infty.
\end{align*}
\end{remark}
\begin{theorem}\label{th6}
Assume that $E=\left(E_{n,k}\right)_{n,k\geq1}$ defines a matrix mapping on $\bigoplus_{n=1}^\infty\mathcal{H}$ and $\bar{E}=\left(\overline{E_{n,k}}\right)_{n,k\geq1}$. Then $\bar{E}$ also defines a matrix mapping on $\bigoplus_{n=1}^\infty\mathcal{H}$. If $E$ is an invertible matrix, then $\bar{E}$ is invertible. If $E$ is a surjective matrix mapping, then $\bar{E}$ acts surjectively on $\bigoplus_{n=1}^\infty\mathcal{H}$.
\end{theorem}
\begin{proof}
Take $\Psi=\left\{\psi_k\right\}_{k=1}\in\bigoplus_{n=1}^\infty\mathcal{H}$ arbitrarily. Then $\sum_{k=1}^\infty E_{n,k}\psi_k$ is convergent for each $n\in\mathbb{N}$ by assumption. Also Lemma \ref{lem3} shows that $\left\{E_{n,k}\right\}_{k=1}^\infty\in\ell^2(\mathbb{N})$. First we prove that $\Psi$ is an $\bar{E}$-sequence. In fact, for each $n\in\mathbb{N}$,
\begin{align*}
\sum_{k=1}^\infty\left\Vert\overline{E_{n,k}}\psi_k\right\Vert=\sum_{k=1}^\infty\left\vert\overline{E_{n,k}}\right\vert\left\Vert\psi_k\right\Vert\leq\left(\sum_{k=1}^\infty\left\vert\overline{E_{n,k}}\right\vert^2\right)^\frac{1}{2}\left(\sum_{k=1}^\infty\left\Vert\psi_k\right\Vert^2\right)^\frac{1}{2}<\infty.
\end{align*}
Hence $\sum_{k=1}^\infty\overline{E_{n,k}}\psi_k$ is absolutely convergent in $\mathcal{H}$, and is therefore convergent. Now we prove that $\bar{E}\Psi\in\bigoplus_{n=1}^\infty\mathcal{H}$. By Remark \ref{rem1}, for fixed $n\in\mathbb{N}$, we have
\begin{align}\label{eq12.1}
	\left\Vert\sum_{k=1}^\infty\overline{E_{n,k}}\psi_k\right\Vert^2=\left\Vert\overline{\sum_{k=1}^\infty\overline{E_{n,k}}\psi_k}\right\Vert^2=\left\Vert\sum_{k=1}^\infty E_{n,k}\overline{\psi_k}\right\Vert^2.
\end{align}
$\bar{\Psi}$ belongs to $\bigoplus_{n=1}^\infty\mathcal{H}$ clearly. Thus $E\bar{\Psi}\in\bigoplus_{n=1}^\infty\mathcal{H}$. Therefore
\begin{align}
\sum_{n=1}^\infty\left\Vert\sum_{k=1}^\infty\overline{E_{n,k}}\psi_k\right\Vert^2=\sum_{n=1}^\infty\left\Vert\sum_{k=1}^\infty E_{n,k}\bar{\psi_k}\right\Vert^2<\infty.
\end{align}
Therefore $\bar{E}\Psi\in\bigoplus_{n=1}^\infty\mathcal{H}$ and $\bar{E}$ is a matrix mapping on $\bigoplus_{n=1}^\infty\mathcal{H}$. 

Now suppose that $E$ is surjective. To show $\bar{E}$ is a surjective matrix mapping, we use the conjugate linear isometry which is defined in Remark \ref{rem1} again. For given $\Psi\in\bigoplus_{n=1}^\infty\mathcal{H}$, the sequence $\bar{\Psi}$ belongs to $\bigoplus_{n=1}^\infty\mathcal{H}$. Since $E$ is a surjective matrix mapping, there exists $\Phi=\left\{\phi_k\right\}_{k=1}\in\bigoplus_{n=1}^\infty\mathcal{H}$ such that $E\Phi=\bar{\Psi}$, that is, for fixed $n\in\mathbb{N}$,
\begin{align}\label{eq14}
\sum_{k=1}^\infty E_{n,k}\phi_k=\bar{\psi_n}.
\end{align}
Note that the series in \eqref{eq14} is convergent. Moreover, the conjugate linear isometry is continuous. Hence
\begin{align}\label{eq15}
\sum_{k=1}^\infty\overline{E_{n,k}}\bar{\phi_k}=\overline{\sum_{k=1}^\infty E_{n,k}\phi_k}=\bar{\bar{\psi_n}}=\psi_n.
\end{align}
\eqref{eq15} shows that $\left(\bar{E}\bar{\Phi}\right)_n=\psi_n$ for all $n\in\mathbb{N}$. Thus $\bar{E}\bar{\Phi}=\Psi$ and $\bar{E}$ is a surjective matrix mapping on $\bigoplus_{n=1}^\infty\mathcal{H}$.

For the rest of the proof, Suppose that $EE^{-1}=E^{-1}E=I$ where $I$ is the infinite identity matrix. This means that
\begin{equation}\label{eq16}
\sum_{k=1}^\infty E_{n,k}E^{-1}_{k,j}=\begin{cases}
0&n\neq j,\\1&n=j.\end{cases}
\end{equation}
If we take the conjugate of both sides of the equality \eqref{eq16}, we obtain
\begin{equation}
\sum_{k=1}^\infty\overline{E_{n,k}}\overline{E^{-1}_{k,j}}=\begin{cases}
0&n\neq j,\\1&n=j.\end{cases}
\end{equation}
which means that $\bar{E}\overline{E^{-1}}=\overline{E^{-1}}\bar{E}=I$. In this way, the argument is completed.
\end{proof}
The next Corollary is a direct result of Theorem \ref{th6}, Corollary \ref{col1} and Theorem \ref{th4}.
\begin{corollary}\label{col2}
Assume that $E=\left(E_{n,k}\right)_{n,k\geq1}$ is an invertible matrix which defines a surjective matrix mapping on $\bigoplus_{n=1}^\infty\mathcal{H}$. Then $\bar{E}=\left(\overline{E_{n,k}}\right)_{n,k\geq1}$ induces a bounded linear bijections on $\bigoplus_{n=1}^\infty\mathcal{H}$ and $\ell^2(\mathbb{N})$.
\end{corollary}
\begin{remark}
Suppose that $E=\left(E_{n,k}\right)_{n,k\geq1}$ defines a matrix mapping on $\bigoplus_{n=1}^\infty\mathcal{H}$ and $\bar{E}=\left(\overline{E_{n,k}}\right)_{n,k\geq1}$. The following statements hold
\begin{itemize}
\item[$(i)$] If $T$ and $\bar{T}$ are the bounded linear operators on $\bigoplus_{n=1}^\infty\mathcal{H}$, which induced by $E$ and $\bar{E}$, respectively, then $\|T\|=\|\bar{T}\|$.
\item[$(ii)$]If $U$ and $\bar{U}$ are the bounded linear operators on $\ell^2(\mathbb{N})$, which induced by $E$ and $\bar{E}$, respectively, then $\|U\|=\|\bar{U}\|$.
\end{itemize}
We only prove $(i)$. The proof of $(ii)$ is similar. 

For given $\Psi\in\bigoplus_{n=1}^\infty\mathcal{H}$, let $\Psi\longrightarrow\bar{\Psi}$ be the conjugate linear isometry we defined in Remark \ref{rem1}. Then for fixed $n\in\mathbb{N}$,
\begin{align*}
\overline{\left(T\Psi\right)_n}=\overline{\sum_{k=1}^\infty E_{n,k}\psi_k}=\sum_{k=1}^\infty\overline{E_{n,k}}\bar{\psi_k}=\left(\bar{T}\bar{\Psi}\right)_n.
\end{align*}
Hence, $\overline{T\Psi}=\bar{T}\bar{\Psi}$. Noting that $\left\|\bar{\Psi}\right\|_{\bigoplus\mathcal{H}}=\left\|\Psi\right\|_{\bigoplus\mathcal{H}}$ for all $\Psi\in\bigoplus_{n=1}^\infty\mathcal{H}$, we have
\begin{align*}
\left\|T\right\|=\sup_{\left\|\Psi\right\|_{\bigoplus\mathcal{H}}\leq1}\left\|T\Psi\right\|_{\bigoplus\mathcal{H}}=\sup_{\left\|\Psi\right\|_{\bigoplus\mathcal{H}}\leq1}\left\|\overline{T\Psi}\right\|_{\bigoplus\mathcal{H}}=\sup_{\left\|\bar{\Psi}\right\|_{\bigoplus\mathcal{H}}\leq1}\left\|\bar{T}\bar{\Psi}\right\|_{\bigoplus\mathcal{H}}=\left\|\bar{T}\right\|.
\end{align*}
\end{remark}
  In the next result, we investigate, when a Bessel sequence is an $E$-Bessel sequence?
\begin{corollary}\label{colEbes}
Let $\mathcal{F}=\lbrace f_{k}\rbrace_{k=1}^{\infty}$ be a Bessel sequence on $\mathcal{H}$ with bound $B$. Suppose that $E=(E_{n,k})_{n,k\geq 1}$ is a matrix mapping on $\bigoplus_{n=1}^\infty\mathcal{H}$. Then $\mathcal{F}$ is an $E$-Bessel sequence with bound $\left\|\bar{E}\right\|^2B$.
\end{corollary}
\begin{proof}
Take $f\in\mathcal{H}$ arbitrarily. Note that $\left\{\left<f,f_k\right>\right\}_{k=1}^\infty\in\ell^2(\mathbb{N})$ since $\mathcal{F}$ is a Bessel sequence. As a consequence of Lemma \ref{lem2}, Theorem \ref{th5} and Theorem \ref{th6}, $\bar{E}$ induces a bounded operator on $\ell^2(\mathbb{N})$. We denote it by $\bar{E}$ again. Since by Lemma \ref{lem3}, $\left\{E_{n,k}\right\}_{k=1}^\infty\in\ell^2(\mathbb{N})$, hence \cite[Corollary 3.2.5]{ch} implies that $\sum_{k=1}^\infty E_{n,k}f_k$ is convergent for all $n\in\mathbb{N}$. This means that $\mathcal{F}$ is an $E$-sequence. Therefore,
 \begin{align*}
 	\sum_{n=1}^{\infty}\left\vert\left\langle f,\left(E\mathcal{F}\right)_{n}\right\rangle\right\vert^{2}&=\sum_{n=1}^{\infty}\left\vert\left\langle f,\sum_{k=1}^{\infty}E_{n,k}f_{k}\right\rangle\right\vert^{2}\\&=\sum_{n=1}^{\infty}\left\vert\sum_{k=1}^{\infty}\overline{E_{n,k}}\left\langle f,f_{k}\right\rangle\right\vert^{2}\\&=\sum_{n=1}^{\infty}\left\vert\left(\bar{E}\left\lbrace\left\langle f,f_{k}\right\rangle\right\rbrace_{k=1}^{\infty}\right)_{n}\right\vert^{2}\\&=\left\Vert\bar{E}\left\lbrace\left\langle f,f_{k}\right\rangle\right\rbrace_{k=1}^{\infty}\right\Vert_{\ell^2}^{2}\\&\leq\left\|\bar{E}\right\|^2\sum_{k=1}^\infty\left|\left<f,f_k\right>\right|^2\leq\left\|\bar{E}\right\|^2B\|f\|^2.
 \end{align*}
\end{proof}
If $T$ is a bounded operator on a Hilbert space $\mathcal{H}$ which represents by a matrix $E=\left(E_{n,k}\right)_{n,k\geq1}$, then for each $n,k\in\mathbb{N}$, $E_{n,k}=\left<Te_k,e_n\right>$ where $\{e_k\}_{k=1}^\infty$ is an orthonormal basis for $\mathcal{H}$. Now if $E^\ast$ is the adjoint of matrix $E$, then for each $n,k\in\mathbb{N}$, we have
\begin{align}\label{eqE*}
E^\ast_{n,k}=\overline{E_{k,n}}=\overline{\left<Te_n,e_k\right>}=\left<e_k,Te_n\right>=\left<T^\ast e_k,e_n\right>.
\end{align}
\eqref{eqE*} shows that $T^\ast$ represents by $E^\ast$. Using this, we state the next lemma as a consequence of Corollary \ref{col1}, Theorem \ref{th6} and Theorem \ref{th4}.
\begin{lemma}\label{lemEt}
Assume that $E=\left(E_{n,k}\right)_{n,k\geq1}$ is an invertible matrix which defines a surjective matrix mapping on $\bigoplus_{n=1}^\infty\mathcal{H}$. Then $E^t=\left(E^t_{n,k}\right)_{n,k\geq1}$, where $E^t_{n,k}=E_{k,n}$ for all $n,k\in\mathbb{N}$, induces a bounded linear bijections on $\bigoplus_{n=1}^\infty\mathcal{H}$ and $\ell^2(\mathbb{N})$.
\end{lemma}
\begin{proof}
$E$ induces a bounded bijection $T$ on $\bigoplus_{n=1}^\infty\mathcal{H}$ by Corollary \ref{col1}. $T^\ast$ is also a bounded bijection and induced by $E^\ast$. Noting that $E^t=\overline{E^\ast}$ and $\left(E^t\right)^{-1}=\left(E^{-1}\right)^t$, $E^t$ induces a bounded bijection on $\bigoplus_{n=1}^\infty\mathcal{H}$ by Theorem \ref{th6}. For the rest of the proof, apply Theorem \ref{th4} to conclude that $E$ induces a bijective operator on $\ell^2(\mathbb{N})$, and then proceed with the argument as in the previous part. 
\end{proof}
\begin{lemma}\label{lemabs}
Suppose that $E=\left(E_{n,k}\right)_{n,k\geq1}$ is a matrix mapping on $\bigoplus_{n=1}^\infty\mathcal{H}$ such that
\begin{align}\label{eq25}
\sum_{n=1}^\infty\left(\sum_{k=1}^\infty\left|E_{n,k}\right|\right)^2<\infty.
\end{align}
Let $\Psi=\{\psi_k\}_{k=1}^\infty$ be a norm bounded above sequence in $\mathcal{H}$. Then for each scalar sequence
$\{c_k\}_{k=1}^\infty\in\ell^2(\mathbb{N})$, the series
\begin{align*}
\sum_{n=1}^\infty\sum_{k=1}^\infty c_nE_{n,k}\psi_k,
\end{align*} 
is absolutely convergent.
\end{lemma}
\begin{proof}
Since $\Psi$ is norm bounded above, thus there exists some $b$ such that $\sup_k\left\|\psi_k\right\|\leq b<\infty.$ Moreover
\begin{align*}
\sum_{n=1}^\infty\sum_{k=1}^\infty\left\|c_nE_{n,k}\psi_k\right\|&=\sum_{n=1}^\infty\left|c_n\right|\sum_{k=1}^\infty\left|E_{n,k}\right|\left\|\psi_k\right\|\\&\leq\left\{\sum_{n=1}^\infty\left|c_n\right|^2\right\}^\frac{1}{2}\left\{\sum_{n=1}^\infty\left(\sum_{k=1}^\infty\left|E_{n,k}\right|\left\|\psi_k\right\|\right)^2\right\}^\frac{1}{2}\\&\leq b\left\{\sum_{n=1}^\infty\left|c_n\right|^2\right\}^\frac{1}{2}\left\{\sum_{n=1}^\infty\left(\sum_{k=1}^\infty\left|E_{n,k}\right|\right)^2\right\}^\frac{1}{2}<\infty.
\end{align*}
\end{proof}
\begin{theorem}\label{Efop}
Let $E$ be a matrix mapping on $\bigoplus_{n=1}^\infty\mathcal{H}$ and $\mathcal{F}=\{f_k\}_{k=1}^\infty$ be a Bessel sequence in $\mathcal{H}$. \begin{itemize}
\item[(i)] For each $f\in\mathcal{H}$, $T_{E\mathcal{F}}^{\ast}f=\bar{E}T_\mathcal{F}^{\ast}f.$
\item[(ii)] If $E$ additionally satisfies \eqref{eq25}, then for each sequence $\lbrace c_n\rbrace_{n=1}^{\infty}\in\ell^2(\mathbb{N})$, 
\begin{equation*}
T_{E\mathcal{F}}\left\lbrace c_n\right\rbrace_{n=1}^{\infty}=T_\mathcal{F}E^t\left\lbrace c_n\right\rbrace_{n=1}^{\infty}.
\end{equation*}
\end{itemize}
\end{theorem}
\begin{proof}
First note that by Corollary \ref{colEbes}, $\mathcal{F}$ is an $E$-Bessel sequence. Thus $T_{E\mathcal{F}}$ and $T^\ast_{E\mathcal{F}}$ are well defined. Also Corollary \ref{col2} states that $\bar{E}$ also defines a bounded operator on $\ell^2(\mathbb{N})$. Now, noting that $\left\lbrace\left\langle f,f_k\right\rangle\right\rbrace_{k=1}^{\infty}$ belongs to $\ell^2(\mathbb{N})$, and using the continuity of the inner product, we obtain 
\begin{align*}
T_{E\mathcal{F}}^{\ast}f&=\left\lbrace\left\langle f,\left(E\mathcal{F}\right)_n\right\rangle\right\rbrace_{n=1}^{\infty}=\left\lbrace\left\langle f,\sum_{k=1}^{\infty}E_{n,k}f_k\right\rangle\right\rbrace_{n=1}^{\infty}\\&=\left\lbrace\sum_{k=1}^{\infty}\overline{E_{n,k}}\left\langle f,f_k\right\rangle\right\rbrace_{n=1}^{\infty}=\bar{E}\left\lbrace\left\langle f,f_k\right\rangle\right\rbrace_{k=1}^{\infty}\\&=\bar{E}T_\mathcal{F}^{\ast}f.
\end{align*}
To prove part $(ii)$, we note that by Lemma \ref{lemEt}, $E^t$ also induces a bounded operator on $\ell^2(\mathbb{N})$. Hence, applying Lemma \ref{lemabs} and the Fubini--Tonelli Theorem, we have
\begin{align*}
T_{E\mathcal{F}}\left\lbrace c_n\right\rbrace_{n=1}^{\infty}&=\sum_{n=1}^{\infty}c_n\left(E\Psi\right)_n=\sum_{n=1}^{\infty}\sum_{k=1}^{\infty
}c_{n}E_{n,k}\psi_k\\&=\sum_{k=1}^{\infty
}\left(\sum_{n=1}^{\infty}c_{n}E_{n,k}\right)\psi_k=T\left\lbrace\sum_{n=1}^{\infty}c_{n}E_{n,k}\right\rbrace_{k=1}^{\infty}\\&=T_\mathcal{F}E^{t}\left\lbrace c_n\right\rbrace_{n=1}^{\infty}.
\end{align*}
\end{proof}
\begin{theorem}\label{thl2plus}
Let $E$ be a matrix mapping on $\ell^2(\mathbb{N})$ and $\mathcal{H}$ be a Hilbert space. Then the following hold
\begin{itemize}
\item[$(i)$] $E$ induces a bounded linear operator on $\bigoplus_{n=1}^\infty\mathcal{H}$.
\item[$(ii)$]If $E$ acts surjectively on $\ell^2(\mathbb{N})$, then the operator induced on $\bigoplus_{n=1}^\infty\mathcal{H}$ is surjective.
\end{itemize}
\end{theorem}
\begin{proof}
\begin{itemize}
\item[$(i)$] By Theorem \ref{th5}, $E$ induces a bounded operator on $\ell^2(\mathbb{N})$ defined by
\begin{align*}
U\{c_k\}_{k=1}^\infty=\left\{\sum_{k=1}^\infty E_{n,k}c_k\right\}_{n=1}^\infty,\qquad\left(\{c_k\}_{k=1}^\infty\in\ell^2(\mathbb{N})\right).
\end{align*}
Hence, Lemma \ref{lem3} implies that $\left\{E_{n,k}\right\}_{k=1}^\infty\in\ell^2(\mathbb{N}).$ Now, for given $\Psi=\left\{\psi_k\right\}_{k=1}^\infty\in\bigoplus_{n=1}^\infty\mathcal{H}$ and each $n\in\mathbb{N}$,
\begin{align}\label{eq17}
\sum_{k=1}^\infty\left\|E_{n,k}\psi_k\right\|=\sum_{k=1}^\infty\left|E_{n,k}\right|\left\|\psi_k\right\|\leq\left(\sum_{k=1}^\infty\left|E_{n,k}\right|^2\right)^\frac{1}{2}\left(\sum_{k=1}^\infty\left\|\psi_k\right\|^2\right)^\frac{1}{2}<\infty.
\end{align}
\eqref{eq17} indicates that $E\Psi$ is well defined sequence in $\mathcal{H}$, for all $\Psi\in\bigoplus_{n=1}^\infty\mathcal{H}$. We also show $E\Psi\in\bigoplus_{n=1}^\infty\mathcal{H}$. To this purpose, let $\left\{e_j\right\}_{j=1}^\infty$ be an orthonormal basis for $\mathcal{H}$. So for each $k\in\mathbb{N}$, we have $\psi_k=\sum_{j=1}^\infty\left<\psi_k,e_j\right>e_j$. Now we fix $m\in\mathbb{N}$, and define the scalar sequence $c^m=\left\{c^m_k\right\}_{k=1}^\infty=\left\{\left<\psi_k,e_m\right>\right\}_{k=1}^\infty,$ which belongs to $\ell^2(\mathbb{N})$, since
\begin{align*}
\sum_{k=1}^\infty\left|\left<\psi_k,e_m\right>\right|^2\leq\sum_{k=1}^\infty\left\|\psi_k\right\|^2<\infty.
\end{align*}
Therefore, $Uc^m\in\ell^2(\mathbb{N})$ and $\left\|Uc^m\right\|_{\ell^2}\leq\left\|U\right\|\left\|c^m\right\|_{\ell^2}$, that is
\begin{align}\label{eq18}
\sum_{n=1}^\infty\left|\sum_{k=1}^\infty E_{n,k}\left<\psi_k,e_m\right>\right|^2\leq\left\|U\right\|^2\sum_{k=1}^\infty\left|\left<\psi_k,e_m\right>\right|^2\leq\left\|U\right\|^2\sum_{k=1}^\infty\left\|\psi_k\right\|^2.
\end{align}
On the other hand, for each $n\in\mathbb{N}$,
\begin{align}\label{eq19}
\left(E\Psi\right)_n=\sum_{k=1}^\infty E_{n,k}\psi_k=\sum_{k=1}^\infty E_{n,k}\left(\sum_{j=1}^\infty\left<\psi_k,e_j\right>e_j\right).
\end{align}
Using \eqref{eq19} and the continuity of the inner product of $\mathcal{H}$ we obtain
\begin{align}\label{eq20}
\left<\left(E\Psi\right)_n,e_m\right>&=\left<\sum_{k=1}^\infty E_{n,k}\left(\sum_{j=1}^\infty\left<\psi_k,e_j\right>e_j\right),e_m\right>\\&=\sum_{k=1}^\infty E_{n,k}\left<\sum_{j=1}^\infty\left<\psi_k,e_j\right>e_j,e_m\right>\nonumber\\&=\sum_{k=1}^\infty E_{n,k}\left<\psi_k,e_m\right>=\left(Uc^m\right)_n\nonumber.
\end{align}
Consequently, by Using \eqref{eq20} and \eqref{eq18} and applying the Fubini--Tonelli Theorem, we have
\begin{align*}
\sum_{n=1}^\infty\left\|\left(E\Psi\right)_n\right\|^2&=\sum_{n=1}^\infty\sum_{m=1}^\infty\left|\left<\left(E\Psi\right)_n,e_m\right>\right|^2=\sum_{n=1}^\infty\sum_{m=1}^\infty\left|\left(Uc^m\right)_n\right|^2\\&=\sum_{m=1}^\infty\sum_{n=1}^\infty\left|\left(Uc^m\right)_n\right|^2=\sum_{m=1}^\infty\left\|Uc^m\right\|_{\ell^2}^2\\&\leq\left\|U\right\|^2\sum_{m=1}^\infty\left\|c^m\right\|_{\ell^2}^2=\left\|U\right\|^2\sum_{m=1}^\infty\sum_{k=1}^\infty\left|\left<\psi_k,e_m\right>\right|^2\\&=\left\|U\right\|^2\sum_{k=1}^\infty\sum_{m=1}^\infty\left|\left<\psi_k,e_m\right>\right|^2=\left\|U\right\|^2\sum_{k=1}^\infty\left\|\psi_k\right\|^2<\infty,
\end{align*}
and so $E$ is a matrix mapping on $\bigoplus_{n=1}^\infty\mathcal{H}$ and by Theorem \ref{th1}, induces a bounded linear operator
\begin{equation*}
T:\bigoplus_{n=1}^\infty\mathcal{H}\longrightarrow\bigoplus_{n=1}^\infty\mathcal{H}\qquad;\qquad T\Psi=\left\{\sum_{k=1}^\infty E_{n,k}\psi_k\right\}_{n=1}^\infty.
\end{equation*}
\item[$(ii)$] As the last part of the proof, we assume that $E$ acts surjectively on $\ell^2(\mathbb{N})$. Then $U$ is surjective obviously and by \cite[Lemma 2.4.1]{ch} $U^\ast$ is bounded below, that is, there exists some $K>0$ such that
\begin{align}\label{eq21}
\left\|U^\ast c\right\|_{\ell^2}\geq K\left\|c\right\|_{\ell^2},\quad\forall c\in\ell^2(\mathbb{N}).
\end{align} 
Since the adjoint operator $U^\ast$ is induced by the matrix $E^\ast$ with $E^\ast_{n,k}=\overline{E_{k,n}}$ for all $n,k\in\mathbb{N}$, hence \eqref{eq21} means that there exists some $K>0$ such that
\begin{align}\label{eq22}
\sum_{n=1}^\infty\left|\sum_{k=1}^\infty\overline{E_{k,n}}c_k\right|^2\geq K\left\|c\right\|_{\ell^2}^2\quad\forall c\in\ell^2(\mathbb{N}).
\end{align} 
Now we show that $T^\ast$ is bounded below. $T^\ast$ is induced by $E^\ast$ as follows
\begin{align*}
T^\ast:\bigoplus_{n=1}^\infty\mathcal{H}\longrightarrow\bigoplus_{n=1}^\infty\mathcal{H}\qquad;\qquad T^\ast\Psi=\left\{\sum_{k=1}^\infty \overline{E_{k,n}}\psi_k\right\}_{n=1}^\infty.
\end{align*}
Thus for each $n\in\mathbb{N}$,
\begin{align}\label{eq23}
\left(T^\ast\Psi\right)_n=\sum_{k=1}^\infty\overline{E_{k,n}}\psi_k=\sum_{k=1}^\infty\overline{E_{k,n}}\left(\sum_{j=1}^\infty\left<\psi_k,e_j\right>e_j\right).
\end{align}
We note that the series in \eqref{eq23} is convergent and the inner product of $\mathcal{H}$ is continuous. Now, If we fix $m\in\mathbb{N}$, then
\begin{align}\label{eq24}
\left<\left(T^\ast\Psi\right)_n,e_m\right>&=\left<\sum_{k=1}^\infty \overline{E_{k,n}}\left(\sum_{j=1}^\infty\left<\psi_k,e_j\right>e_j\right),e_m\right>\\&=\sum_{k=1}^\infty \overline{E_{k,n}}\left<\sum_{j=1}^\infty\left<\psi_k,e_j\right>e_j,e_m\right>=\sum_{k=1}^\infty \overline{E_{k,n}}\left<\psi_k,e_m\right>=\left(U^\ast c^m\right)_n\nonumber.
\end{align}
Therefore, using \eqref{eq24} and \eqref{eq22} and the Fubini--Tonelli Theorem, we conclude that
\begin{align*}
\left\|T^\ast\Psi\right\|_{\bigoplus\mathcal{H}}^2=\sum_{n=1}^\infty\left\|\left(T^\ast\Psi\right)_n\right\|^2&=\sum_{n=1}^\infty\sum_{m=1}^\infty\left|\left<\left(T^\ast\Psi\right)_n,e_m\right>\right|^2=\sum_{n=1}^\infty\sum_{m=1}^\infty\left|\left(U^\ast c^m\right)_n\right|^2\\&=\sum_{m=1}^\infty\sum_{n=1}^\infty\left|\left(U^\ast c^m\right)_n\right|^2=\sum_{m=1}^\infty\left\|U^\ast c^m\right\|_{\ell^2}^2\\&\geq K\left\|U^\ast\right\|^2\sum_{m=1}^\infty\left\|c^m\right\|_{\ell^2}^2=K\left\|U^\ast\right\|^2\sum_{m=1}^\infty\sum_{k=1}^\infty\left|\left<\psi_k,e_m\right>\right|^2\\&=K\left\|U^\ast\right\|^2\sum_{k=1}^\infty\sum_{m=1}^\infty\left|\left<\psi_k,e_m\right>\right|^2=K\left\|U^\ast\right\|^2\left\|\Psi\right\|_{\bigoplus\mathcal{H}}^2.
\end{align*}
In this way, we proved that $T^\ast$ is bounded below, and by \cite[Lemma 2.4.1]{ch}, $T$ is therefore surjective.
\end{itemize}
\end{proof}
Theorem \ref{thl2plus} and Corollary \ref{col1}, lead immediately to the following result. 
\begin{corollary}
If $E$ is an invertible matrix which define a surjective matrix mapping on $\ell^2(\mathbb{N})$, then $E$ induces a bounded bijection $T$ on $\bigoplus_{n=1}^\infty\mathcal{H}$. Furthermore, $T^{-1}$ is bounded and $E^{-1}$ represents $T^{-1}$.
\end{corollary}
\begin{theorem}
Let $\mathcal{F}=\left\{f_k\right\}_{k=1}^\infty$ be a Riesz basis for $\mathcal{H}$ and $E=\left(E_{n,k}\right)_{n,k\geq1}$ be it's associated Gramm matrix, that is $E_{n,k}=\left<f_n,f_k\right>$ for all $n,k\in\mathbb{N}$. Then $\mathcal{F}$ is an $E$-frame.
\end{theorem}
\begin{proof}
By \cite[Theorem 3.6.6]{ch}, $E$ induces a bounded invertible operator on $\ell^2(\mathbb{N})$. As we showed in Theorem \ref{th3}, $\mathcal{F}$ is an $E$-sequence and the sequence $E\mathcal{F}$ satisfies the upper and lower frame bounds. Also by Theorem \ref{thl2plus}, $E$ can be considered as a matrix mapping on $\bigoplus_{n=1}^\infty\mathcal{H}$. Therefore, the conditions for defining $E$-frame are met.
\end{proof}
\begin{lemma}\label{lemEsi-si}
Suppose that $\Psi=\left\lbrace\psi_k\right\rbrace_{k=1}^\infty$ is a norm bounded above sequence in $\mathcal{H}$ and $E=\left(E_{n,k}\right)_{n,k\geq1}$ is an infinite matrix mapping on $\bigoplus_{n=1}^\infty\mathcal H$ such that 
\begin{equation}\label{eql1l2}
\sum_{n=1}^\infty\left(\sum_{k=1}^\infty\left|E_{n,k}-\delta_{n,k}\right|\right)^2\leq M<\infty.
\end{equation}
Then $\left\lbrace\left(E\Psi\right)_n-\psi_n\right\rbrace_{n=1}^\infty\in\bigoplus_{n=1}^\infty\mathcal{H}$.
\end{lemma}
\begin{proof}
First we show that $E\Psi$ is a well defined sequence in $\mathcal{H}$. 
To this end, we fixed $n\in\mathbb{N}$. Now, by using \eqref{eql1l2} we have
\begin{align*}
\left(\sum_{k=1}^\infty\left|E_{n,k}-\delta_{n,k}\right|\right)^2\leq\sum_{n=1}^\infty\left(\sum_{k=1}^\infty\left|E_{n,k}-\delta_{n,k}\right|\right)^2\leq M<\infty.
\end{align*}
Hence
\begin{align*}
\sum_{k=1}^\infty\left|E_{n,k}\right|-\sum_{k=1}^\infty\left|\delta_{n,k}\right|\leq\sum_{k=1}^\infty\left|E_{n,k}-\delta_{n,k}\right|\leq\sqrt{M}<\infty.
\end{align*}
Therefore, $\sum_{k=1}^\infty\left|E_{n,k}\right|\leq\sqrt{M}+1<\infty$. By assumption there exists a constant $b$ such that $\sup_k\|\psi_k\|\leq b<\infty$. Thus we have
\begin{align}\label{eqabs}
\sum_{k=1}^\infty\left\|E_{n,k}\psi_k\right\|=\sum_{k=1}^\infty\left|E_{n,k}\right|\left\|\psi_k\right\|\leq b\sum_{k=1}^\infty\left|E_{n,k}\right|<\infty.
\end{align}
\eqref{eqabs} shows that $\sum_{k=1}^\infty E_{n,k}\psi_k$ is an absolutely convergent series in the Hilbert space $\mathcal{H}$ and is therefore convergent in $\mathcal{H}$. 

Moreover,
\begin{align}
\sum_{n=1}^\infty\left\Vert\left(E\psi\right)_n-\psi_n\right\Vert^2&=\sum_{n=1}^\infty\left\Vert\sum_{k=1}^\infty E_{n,k}\psi_k-\psi_n\right\Vert^2\\&=\sum_{n=1}^\infty\left\Vert\sum_{k=1}^\infty\left(E_{n,k}-\delta_{n,k}\right)\psi_k\right\Vert^2\nonumber\\&\leq\sum_{n=1}^\infty\left(\sum_{k=1}^\infty\left\vert E_{n,k}-\delta_{n,k}\right\vert\left\Vert\psi_k\right\Vert\right)^2\nonumber\\&\leq b^2\sum_{n=1}^\infty\left(\sum_{k=1}^\infty\left|E_{n,k}-\delta_{n,k}\right|\right)^2\leq b^2M<\infty\nonumber. 
\end{align}
\end{proof}
\begin{theorem}\label{thpert}
Suppose that $\Psi=\left\lbrace\psi_k\right\rbrace_{k=1}^\infty$ is a frame for $\mathcal{H}$ with bounds $A,B$, $E=\left(E_{n,k}\right)_{n,k\geq1}$ is an infinite matrix that satisfies \eqref{eql1l2} and $M$ is the constant appearing therein.
 If $M<\frac{A}{B}$, then $E\Psi$ is a frame for $\mathcal{H}$ with bounds $$A\left(1-\sqrt{\frac{BM}{A}}\right)^2\quad,\quad B\left(1+\sqrt{M}\right)^2.$$
In fact, $\Psi$ is an $E$-frame, since \eqref{eql1l2} implies that $E$ is an invertible matrix mapping on $\bigoplus_{n=1}^\infty\mathcal{H}$. 
\end{theorem}
\begin{proof}
Since $\Psi$ is a frame with upper bound $B$, hence $\sup_k\left\|\psi_k\right\|\leq\sqrt{B}$. 
As we showed in Lemma \ref{lemEsi-si}, $\Psi$ is an $E$-sequence and $\left(E\Psi-\Psi\right)\in\bigoplus_{n=1}^\infty\mathcal{H}$. Thus for all $\Psi\in\bigoplus_{n=1}^\infty\mathcal{H}$, $E\Psi=\left(E\Psi-\Psi\right)+\Psi\in\bigoplus_{n=1}^\infty\mathcal{H}$. Also $E$ is an invertible operator. Indeed, by applying the Fubini--Tonelli Theorem, for all $c=\{c_k\}_{k=1}^\infty\in\ell^2(\mathbb{N})$ we have
\begin{align}\label{eqEc-Ic}
	\left\|Ec-Ic\right\|_{\ell^2}^2=\sum_{n=1}^\infty\left|\left(Ec\right)_n-c_n\right|^2&=\sum_{n=1}^\infty\left|\sum_{k=1}^\infty E_{n,k}c_k-\delta_{n,k}c_k\right|^2\\&\leq\sum_{n=1}^\infty\left(\sum_{k=1}^\infty\left|E_{n,k}-\delta_{n,k}\right||c_k|\right)^2\nonumber\\&=\sum_{n=1}^\infty\left(\sum_{k=1}^\infty\left|E_{n,k}-\delta_{n,k}\right|^{\frac{1}{2}}\left|E_{n,k}-\delta_{n,k}\right|^{\frac{1}{2}}|c_k|\right)^2\nonumber\\&\leq\sum_{n=1}^\infty\sum_{k=1}^\infty\left|E_{n,k}-\delta_{n,k}\right|\sum_{j=1}^\infty\left|E_{n,j}-\delta_{n,j}\right||c_j|^2\nonumber\\&=\sum_{j=1}^\infty|c_j|^2\sum_{n=1}^\infty\sum_{k=1}^\infty\left|E_{n,k}-\delta_{n,k}\right|\left|E_{n,j}-\delta_{n,j}\right|\nonumber,
\end{align}
where $I$ is the identity operator on $\ell^2(\mathbb{N})$.

On the other hand, it is clear that for each $n,j\in\mathbb{N}$, $\left|E_{n,j}-\delta_{n,j}\right|\leq\sum_{k=1}^\infty\left|E_{n,k}-\delta_{n,k}\right|$. Thus
\begin{align}\label{eq29}
\sum_{k=1}^\infty\left|E_{n,k}-\delta_{n,k}\right|	\left|E_{n,j}-\delta_{n,j}\right|\leq\left(\sum_{k=1}^\infty\left|E_{n,k}-\delta_{n,k}\right|\right)^2,
\end{align}
Now we sum \eqref{eq29} from $n=1$ to $\infty$ and use \eqref{eql1l2} to obtain the following
\begin{align*}
	\sum_{n=1}^\infty\sum_{k=1}^\infty\left|E_{n,k}-\delta_{n,k}\right|	\left|E_{n,j}-\delta_{n,j}\right|\leq\sum_{n=1}^\infty\left(\sum_{k=1}^\infty\left|E_{n,k}-\delta_{n,k}\right|\right)^2\leq M.
\end{align*}
Therefore \eqref{eqEc-Ic} implies that $\left\|Ec-Ic\right\|_{\ell^2}^2\leq M\|c\|_{\ell^2}^2$. But $BM<A$ by assumption. Hence
\begin{equation*}
	\left\|E-I\right\|\leq\sqrt{M}<\sqrt{\frac{A}{B}}\leq1,
\end{equation*}
and $E$ is an invertible operator, that is, it has an inverse matrix.
 Moreover, using Lemma \ref{lemEsi-si}, for given $f\in\mathcal{H}$, we have
\begin{align*}
\sum_{n=1}^\infty\left\vert\left\langle f,\left(E\Psi\right)_n-\psi_n\right\rangle\right\vert^2&\leq\sum_{n=1}^\infty\left\Vert f\right\Vert^2\left\Vert\left(E\Psi\right)_n-\psi_n\right\Vert^2\\&\leq BM\left\Vert f\right\Vert^2.
\end{align*}
As $BM<A$, the remainder of the proof is obtained by applying \cite[Corollary 22.1.5]{ch}.
\end{proof}



\begin{thebibliography}{0}

\bibitem{CA} P.G. Casazza, \emph{The art of frame theory}, Taiwan. J. Math., \textbf{4}(2) (2000), 129-–201.

\bibitem{ch} O. Christensen, \emph{An introduction to frames and Riesz bases}, Birkhauser, Boston, 2016.

\bibitem{Conway} J. B. Conway, {\it A Course in Functional Analysis,}
2nd Edition, Springer-Verlag, 1990.

\bibitem{DG} I. Daubechies, A. Grassman and Y. Meyer, \emph{Painless nonothogonal expanisions}, J. Math. Phys.,  
\textbf{27} (1986), 1271--1283.

\bibitem{DS} R.J. Duffin and A.C. Schaeffer, \emph{A class of nonharmonic Fourier series}, Trans. Amer.
Math. Soc., \textbf{72} (1952), 341--366.

\bibitem{HAN} D. Han and D. Larson, \emph{Frame, Bases and group representations}, Memoir. Amer. Math. Soc. \textbf{147} (2000), 1–-94.

\bibitem{TAL} G. Talebi, M. A. Dehghan, \textit{On $E$-frames in separable Hilbert spaces},
Banach J. Math. Anal. 9(3) (2015), 43--74.



\end{thebibliography}
\end{document}